\documentclass{raex}
\usepackage[T1]{fontenc}

\usepackage{amsmath, amssymb, amsthm, amsfonts, amsbsy}
\usepackage{mathrsfs}
\usepackage{graphicx}
\usepackage{tikz}
\usepackage{enumitem, lineno}
\usepackage{makeidx}
\usepackage{textcomp}
\usepackage{comment}
\usepackage{diagmac2}
\usepackage{color}
\usepackage{cite}

\usetikzlibrary{
    positioning,
    calc,
    trees,
    decorations.markings,
    patterns,
    arrows,
    arrows.meta
}

\usepackage[
    colorlinks=true,
    linkcolor=red,
    citecolor=cyan,
    urlcolor=blue,
    breaklinks=true
]{hyperref}

 \begin{Author}
	\FirstName{Oleksiy}\LastName{Dovgoshey}
	\PostalAddress{Department of Function Theory, Institute of Applied Mathematics and Mechanics of NASU, Cherkasy, 18031, Ukraine;\\
Department of Mathematics and Statistics, University of Turku, Turku, 20014, Finland}
	\Email{oleksiy.dovgoshey@gmail.com, oleksiy.dovgoshey@utu.fi }
\end{Author}

\begin{Author}
	\FirstName{Navjay Singh}\LastName{Jaiswal}
	\PostalAddress{Department of Mathematics, IIT, Gandhinagar, 382055, India}
	\Email{navjayjaiswal@iitgn.ac.in}
\end{Author}

\begin{Author}
	\FirstName{Surinder Pal Singh}\LastName{Kainth$^*$}

	\PostalAddress{Department of Mathematics, Panjab University, Chandigarh, 160014, India}
	\Email{sps@pu.ac.in}
\end{Author}

\begin{Author}
	\FirstName{Gursimar}\LastName{Kaur}
	\PostalAddress{Department of Mathematics, Panjab University, Chandigarh, 160014, India}
	\Email{asbhangu143@gmail.com}
\end{Author}

 \begin{Author}
	\FirstName{Olga}\LastName{Rovenska}
	\PostalAddress{Department of Mathematics, Applied Mathematics and Physics, Donbas State Engineering Academy,  Kramatorsk, 84313, Ukraine}
	\Email{rovenskaya.olga.math@gmail.com}
\end{Author}

\begin{MathReviews}
	\primary{54E35}
    \secondary{05C05}
\end{MathReviews}

\begin{KeyWords}
\end{KeyWords}

 \newtheorem{thm}{Theorem}[section]
  \newtheorem{defn}[thm]{Definition}
 \newtheorem{cor}[thm]{Corollary}
 \newtheorem{lem}[thm]{Lemma} 
 \newtheorem{prop}[thm]{Proposition} 
 \newtheorem{prob}[thm]{Problem}
  \newtheorem{con}[thm]{Conjecture}

 \theoremstyle{definition} 

 \theoremstyle{remark}
 \newtheorem{rem}[thm]{Remark}

 \numberwithin{equation}{section}

\newtheorem{example}[thm]{Example}

\begin{document}\large
\setlist[enumerate,1]{label={(\roman*)}}

\title{Maximal Center of Distances of Finite Ultrametric Spaces and Perfect Binary Trees}

\maketitle

 \begin{abstract} 
We investigate the finite ultrametric spaces $(X,d)$ that have a given cardinality of the center of distances and a minimal cardinality of the set $X$. It is shown that such spaces are isometric if and only if their centers of distances are the same.
The representing trees of these spaces are characterized up to isomorphism.
\end{abstract}




\section{Introduction}

The concept of the center of distances was introduced by Wojciech Bielaś, Szymon Plewik, and Marta Walczyńska in  paper \cite{BPW2018} as follows.
\begin{defn}\label{zcgh572}
Let $(X,d)$ be a metric space and let $D(X)$ be the distance set of $(X,d)$,
\begin{equation*}
D(X) := \{ d(x,y) : x,y \in X \}.
\end{equation*}
The {\it center of distances} $C(X)$  of $(X,d)$ is the set of all $t\in D(X)$ such that for each $p\in X$  there exists $x\in X$ with $d(p,x)=t$.
\end{defn}

In \cite{BPW2018}, the concept of the center of distances was used to generalize the theorem of John von Neumann on permutations of two sequences having the same set of cluster points.
The center of distances has turned out to be an effective tool for solving a number of interesting problems connected with metric spaces (see, e.g., papers
\cite{Bartoszewicz2018,Banakiewicz2022,Banakiewicz2023,Kula2025,NowakowskiPrusWisniowski2026,Bartoszewicz2024}).

In the case of ultrametric spaces, the structure of the center of distances was recently studied in \cite{DR2026MDPI} and \cite{DovgosheyRovenska2026}.

The following theorem was  proved in \cite{DovgosheyRovenska2026}.

\begin{thm}\label{ter1-1}
Let $(X,d)$ be a finite ultrametric space and let $n:=|X|$.
Then the inequality
\begin{equation*}
|C(X)| \leq 1 + \lfloor \log_2 (n) \rfloor
\end{equation*}
holds, where  $\log_2(n)$ is the binary logarithm of $n$ and  $\lfloor \log_2(n) \rfloor$ is the integer part of $\log_2(n)$. 
Moreover, for every integer $n\geq 1$ there exists an ultrametric space $(Y,\rho)$ such that $|Y|=n$
and
\begin{equation*}
|C(Y)| = 1 + \lfloor \log_2(n) \rfloor .
\end{equation*}
\end{thm}

Theorem \ref{ter1-1} implies the following corollary.
\begin{cor}
The inequality
\[
|X| \geq 2^{|C(X)|-1}
\]
holds for each finite ultrametric space $(X,d)$.
\end{cor}

The present paper is motivated by the following problem.

\begin{prob}\label{qwed}
Let $\bf MCD$ (maximal center of distances) denote the class of all finite ultrametric spaces $(X,d)$ satisfying the equality
\begin{equation*}
|X| = 2^{|C(X)|-1}. 
\end{equation*}
Describe the $\bf MCD$-spaces up to isometry.
\end{prob}

It was shown by Vladimir~Gurvich and Mikhail~Vyalyi that any finite ultrametric space is isometrically describable by a rooted tree equipped with a specific labeling of the vertex set \cite{GV2012DAM}.
 The corresponding geometric interpretation of the
Gurvich-Vyalyi representation was obtained in  \cite{PD2014JMS} and it provides a framework for addressing various
extremal problems related to finite ultrametric spaces \cite{DPT2015,DPT2017FPTA,DP2020pNUAA}. Analogues of the
Gurvich-Vyalyi representation and its geometric interpretation have recently been
obtained for totally bounded ultrametric spaces \cite{Dovg2025}.

The present paper is organized as follows. Section~\ref{seq2} contains some definitions and facts from the theory of ultrametric spaces and graph theory. Section~\ref{seq3} collects several auxiliary lemmas.

The main results of the paper are proved in Section \ref{seq4}.
The structure of representing trees of {\bf MCD}-spaces is described in Theorem~\ref{mmre} up to isomorphisms of labeled rooted trees. The necessary and sufficient conditions under which {\bf MCD}-spaces are isometric are given in Theorem~\ref{m0}.

In the final Section \ref{seq5}, we formulate some conjectures and open problems related to ${\bf MCD}$-spaces and finite homogeneous ultrametric spaces.

\section{Preliminaries}\label{seq2}

 In what follows we denote by $\mathbb{R}^{+}$ the set $[0,\infty)$,
\(
\mathbb{R}^{+}:=[0,\infty).
\)

\begin{defn}\label{ganb}
Let $X$ be a non-empty set. A {\it metric} on  $X$ is a function $d : X \times X \to \mathbb{R}^+$,  such that for all $x,y,z \in X$:
\begin{enumerate}
\item[{\it(i)}] $d(x,y)=d(y,x)$,
\item[{\it(ii)}]  $d(x,y)=0 \iff x=y$,
\item[{\it(iii)}]  $d(x,y)\le d(x,z)+d(z,y)$.
\end{enumerate}
\end{defn}

If the inequality
\begin{equation}
    \label{ddd}
d(x,y)\le \max \{d(x,z),d(z,y)\}
\end{equation}
holds instead of $(iii)$, then $d : X \times X \to \mathbb{R}^+$ is called an {\it ultrametric} on $X$.

Inequality \eqref{ddd} is often called the {\it strong triangle inequality}. See, for example, book \cite[p. 40]{Surinder}.

Let $\mathrm{S} $ be a non-empty subset of a metric space $(X,d)$. The quantity
\[
\operatorname{diam}\mathrm{S}
:= \sup\{\, d(x,y) : x,y \in \mathrm{S} \,\}
\]
is called the {\it diameter} of $\mathrm{S}$.

\begin{prop}\label{xxtt}
 Let $S$ be a non-empty subset of an
ultrametric space $(X,d)$. Then the equality
\[
\operatorname{diam} S=\sup\{d(p,s):s\in \mathrm{S}\}
\]
holds for each $p\in \mathrm{S}$.
\end{prop}

For the proof of Proposition \ref{xxtt} see, for example, Lemma 1 in \cite{DR2026MDPI}.

Let $(X,d)$ and $(Y,\rho)$ be metric spaces. A mapping 
\(
\Phi : X \to Y
\)
is called an {\it isometry} of $(X,d)$ and $(Y,\rho)$ if this mapping is bijective and the equality
\[
d(x,y)=\rho\bigl(\Phi(x),\Phi(y)\bigr)
\]
holds for all $x,y\in X$.
We say that $(X,d)$ and $(Y,\rho)$ are {\it isometric} if there exists an isometry $X \to Y$.

Let us turn now to the graph theory.
  Recall that a {\it graph} is a pair $(V,E)$ consisting of a non-empty set $V$ and a (possibly empty) set $E$ whose elements are unordered pairs of distinct points from $V$.
For a graph $G=(V,E)$, the sets $V=V(G)$ and $E=E(G)$ are called the {\it set of vertices} and the {\it set of edges}, respectively. If $x$ and $y$ are vertices of $G$ and $\{x,y\}\in E(G)$, then we say that $x$ and $y$ are {\it adjacent} in $G$. A graph $G$ is said to be {\it empty} if the equality $E(G)=\emptyset$ holds. 
A graph $G$ is said to be {\it finite} if this graph contains a finite number of vertices, $|V(G)| < \infty$.
A graph $H$ is called a {\it subgraph} of a graph $G$, $H \subseteq G$, if 
\[
V(H) \subseteq V(G) \quad \text{and} \quad E(H) \subseteq E(G).
\]

\begin{defn}
    \label{nnvv}
 Let $G$ be a graph and let $V'$ be a non-empty
subset of $V(G)$. A subgraph of $G$ is called the subgraph
induced by $V'$ and denoted by $G[V']$ if the equality
\(
V(G[V']) = V'
\)
holds, and the equivalence
\[
\left(\{u',v'\} \in E(G[V']) \right) \iff
\left(\{u',v'\} \in E(G)\right)
\]
is valid for all $u',v' \in V'$.
\end{defn}

Let $v$ be a vertex of a graph $G$. The cardinal number of the set
\[
\{u \in V(G) : \{u,v\} \in E(G)\}
\]
is called the {\it degree} of $v$ and denoted by $\delta_G(v)$,
\begin{equation}
    \label{mmks}
\delta_G(v) := \left|\{u \in V(G) : \{u,v\} \in E(G)\}\right|.
\end{equation}

A {\it path} is a graph $P$ of the form
\begin{equation}
    \label{gasl}
V(P)=\{x_0,x_1,\dots,x_k\}, \quad E(P)=\{\{x_0,x_1\},\dots,\{x_{k-1},x_k\}\},
\end{equation}
where all $x_i$ are distinct and $k \geq 1$. If \eqref{gasl} holds, then we write $P=P_{x_0,x_k}$ and say that $P$ is a path joining the vertices $x_0$ and  $x_k$.

\begin{defn}\label{miy}
A graph $G$ is called {\it connected} if for every two distinct vertices
$u,v\in V(G)$ there exists a path
\(
P_{u,v}\subseteq G.
\)
\end{defn}

A connected subgraph $H$ of a graph $G$ is  a {\it component} of $G$ if the induced subgraph 
\(
G\bigl[V(H)\cup\{v\}\bigr]
\)
is not connected whenever
\(
v\in V(G)\setminus V(H).
\)

A graph $C$ is called a {\it cycle} if $|V(C)|\ge 3$ and there exists an enumeration $v_1,v_2,\dots,v_n$ of its vertices such that
\[
\{v_i,v_j\}\in E(C) \iff (|i-j|=1 \text{ or } |i-j|=n-1).
\]

\begin{defn}\label{avct}
    A connected acyclic graph is called a {\it tree}.
\end{defn}

The next proposition describes a characteristic property of trees.
See, for example, Theorem 1.5.1   in~\cite{Diestel2017}.

\begin{prop}\label{hfdsxbj} Let $G$ be a finite  graph. Then the following statements are equivalent.
\begin{enumerate}
\item[(i)] $G$ is a tree.
\item[(ii)] Any two distinct vertices of $G$ are joined by unique path in $G$.
\end{enumerate}
\end{prop}

A tree $T$ may have a distinguished vertex $r$  the {\it root} of $T$; in this case $T$ is called a {\it rooted tree} and we write $T=T(r)$.

Let $T=T(r)$ be a rooted tree and let $v\in V(T)$. Similarly \cite{Dov2019pNUAA,DP2018pNUAA,DP2019PNUAA,DP2020pNUAA,DPT2015,DPT2017FPTA}, we will denote by $\delta_{T(r)}^{+}(v)$ the {\it out-degree} of $v$,
\begin{equation}
    \label{zbh1}
\delta_{T(r)}^{+}(v):=
\begin{cases}
\delta_T(v), & \text{if } v=r,\\
\delta_T(v)-1, & \text{if } v\ne r,
\end{cases}
\end{equation}
where $\delta_T(v)$ is the degree of $v$ defined by \eqref{mmks} with $G=T$.

\begin{defn}
    \label{scm}
A vertex $v$ of a rooted tree $T=T(r)$ is called a {\it leaf} of $T(r)$ if the equality $\delta^+_{T(r)}(v)=0$ holds. In what follows we denote by $L(T(r))$ the set of leaves of the rooted tree $T(r)$.
\end{defn}

The vertices of a rooted tree $T=T(r)$ that are not its leaves are called the {\it internal nodes} of $T(r)$.

If $T = T(r)$ is a rooted tree and $v$ is a vertex of $T$ such that $v \neq r$, 
then the {\it level} of $v$ is the number 
\begin{equation}\label{lev}
    \operatorname{lev}_{T(r)}(v):=|E(P_{v,r})|,
\end{equation}
where $P_{v,r}$ is the unique path joining $v$ with $r$.
Moreover, the root $r$ of $T(r)$ has, by definition,  the zero level, 
\begin{equation}
    \label{p1}
    \operatorname{lev}_{T(r)}(r):=0.
\end{equation}

\begin{defn}\label{sqbh}
Let $T=T(r)$ be a rooted tree, and let $v\in V(T)\setminus \{r\}$. A vertex $u$ is called 
the {\it direct predecessor}
of $v$ if $u \in V(P_{v,r})$ and 
\begin{equation*}
\operatorname{lev}_{T(r)}(v)=1+\operatorname{lev}_{T(r)}(u).
\end{equation*}
In this case we say that $v$ is a {\it direct successor}
of $u$.
\end{defn}

Proposition \ref{hfdsxbj} implies that each vertex $v\in V(T)\setminus \{r\}$ has the unique direct predecessor. Moreover, the root $r$ is the only vertex of  $T$ that has no any direct predecessor.

The next theorem directly follows from Theorem 3.38 of \cite[p.~20]{Tutte1966}.

\begin{thm}
    \label{gjl}
Let $T=T(r)$ be a rooted tree with $\delta_T(r)\geq 1$ and let
$T-r$ be the subgraph of $T$ induced by the set
$V(T)\setminus\{r\}$. Then $T-r$ has exactly
$\delta_T(r)$ components and all these components are trees.
\end{thm}

We will use the concept of isomorphic rooted trees. Before introducing this concept into consideration, it is useful to remind the definition of isomorphism for free trees.
\begin{defn}\label{r1}
Let $T_1$ and $T_2$ be trees. A mapping
\(
F : V(T_1) \rightarrow V(T_2)
\)
is an isomorphism of $T_1$ and $T_2$ if $F$ is bijective and the equivalence
\[
\{u,v\}\in E(T_1)
\iff
\{F(u),F(v)\}\in E(T_2)
\]
is valid for all $u,v\in V(T_1)$. Two trees are isomorphic if there exists an isomorphism of these trees.
\end{defn}

For the case of rooted trees, Definition \ref{r1} must be modified as follows.

\begin{defn}\label{wf}
Let $T_1=T_1(r_1)$ and $T_2=T_2(r_2)$ be rooted trees. A mapping
\(
F:V(T_1)\to V(T_2)
\)
is called an isomorphism of $T_1(r_1)$ and $T_2(r_2)$ if $F$ is an isomorphism of free trees $T_1$ and $T_2$ and the equality 
\(
    F(r_1)=r_2
\)
holds.
\end{defn}

Let us define now the concepts of binary trees and perfect binary trees.

\begin{defn}\label{xxetr}
A {\it binary} tree is a rooted tree $T=T(r)$ such that each vertex $v$ of $T$ is either a leaf of $T(r)$ or has the out-degree two, $\delta^+_{T(r)}(v)=2.$
A binary tree $T(r)$ is called {\it perfect} if the equality 
\begin{equation*}
    \operatorname{lev}_{T(r)}(v)=   \operatorname{lev}_{T(r)}(u)
\end{equation*}
holds for all $u,v \in L(T(r))$.
\end{defn}

\begin{rem}
The term ``binary tree'' used in Definition~\ref{xxetr} is not universally accepted. Steven Finch calls such trees
\emph{strictly binary} in \cite[p.~298]{Finch2003}, these trees are also called \emph{full binary} by Allan Bickle in \cite[p.~56]{Bickle2020}.
We use
the term ``binary
tree'' in accordance with Michael
Drmota's terminology, see, for example,
\cite{Drmota2004,Drmota2009}.
\end{rem}

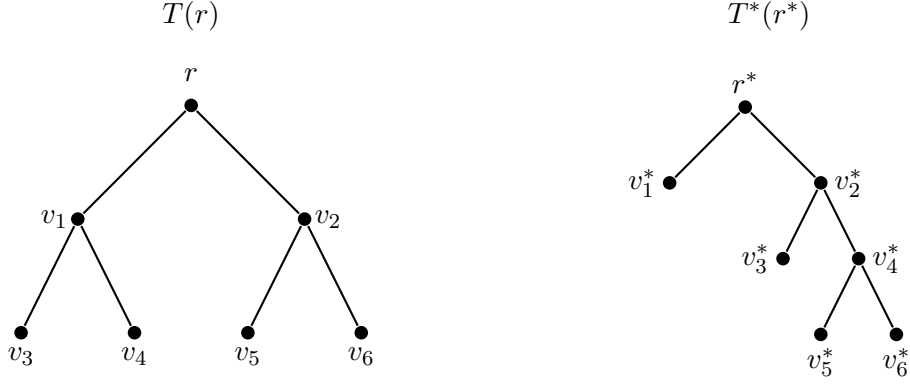
\begin{figure}[h]
\centering

\begin{minipage}[t]{0.46\textwidth}
\centering
$T(r)$

\begin{tikzpicture}[
baseline=(current bounding box.north),
every node/.style={circle,fill=black,inner sep=1.8pt},
every path/.style={thick}
]

\node (r) at (0,4.5) {};
\node[draw=none,fill=none] at (0,4.9) {$r$};
\node (a) at (-1.5,3) {};
\node (b) at ( 1.5,3) {};
\node (c) at (-2.25,1.5) {};
\node (d) at (-0.75,1.5) {};
\node (e) at ( 0.75,1.5) {};
\node (f) at ( 2.25,1.5) {};

\draw (r)--(a);
\draw (r)--(b);
\draw (a)--(c);
\draw (a)--(d);
\draw (b)--(e);
\draw (b)--(f);

\node[draw=none,fill=none,left]  at (a) {$v_1$};
\node[draw=none,fill=none,right] at (b) {$v_2$};
\node[draw=none,fill=none,below] at (c) {$v_3$};
\node[draw=none,fill=none,below] at (d) {$v_4$};
\node[draw=none,fill=none,below] at (e) {$v_5$};
\node[draw=none,fill=none,below] at (f) {$v_6$};

\end{tikzpicture}
\end{minipage}
\hfill
\begin{minipage}[t]{0.46\textwidth}
\centering
$T^{*}(r^{*})$

\begin{tikzpicture}[
baseline=(current bounding box.north),
every node/.style={circle,fill=black,inner sep=1.8pt},
every path/.style={thick}
]

\node (r) at (0,3) {};
\node[draw=none,fill=none,above] at (r) {$r^{*}$};
\node (a) at (-1,2) {};
\node (b) at ( 1,2) {};
\node (c) at (0.5,1) {};
\node (d) at (1.5,1) {};
\node (e) at (1.0,0) {};
\node (f) at (2.0,0) {};

\draw (r)--(a);
\draw (r)--(b);
\draw (b)--(c);
\draw (b)--(d);
\draw (d)--(e);
\draw (d)--(f);

\node[draw=none,fill=none,left]  at (a) {$v_1^{*}$};
\node[draw=none,fill=none,right] at (b) {$v_2^{*}$};
\node[draw=none,fill=none,left]  at (c) {$v_3^{*}$};
\node[draw=none,fill=none,right] at (d) {$v_4^{*}$};
\node[draw=none,fill=none,below] at (e) {$v_5^{*}$};
\node[draw=none,fill=none,below] at (f) {$v_6^{*}$};

\end{tikzpicture}
\end{minipage}

\caption{Non-isomorphic binary  trees $T=T(r)$ and $T^*=T^{*}(r^{*})$ have the same number of vertices,
the same number of leaves, and the same number of internal nodes.\label{fig1}}
\end{figure}

\begin{example}
Let us consider the binary  trees $T=T(r)$ and $T^*=T^*(r^{*})$ depicted in Figure~\ref{fig1}. 
The binary tree $T = T(r)$ is
perfect, since
\[
\operatorname{lev}_{T(r)}(v_3)
=
\operatorname{lev}_{T(r)}(v_4)
=
\operatorname{lev}_{T(r)}(v_5)
=
\operatorname{lev}_{T(r)}(v_6)
=
2,
\]
but $T^* = T^*(r^*)$ is not perfect,
\[
\operatorname{lev}_{T^*(r^*)}(v_1^*) = 1, \quad
\operatorname{lev}_{T^*(r^*)}(v_3^*) = 2,
\quad 
\operatorname{lev}_{T^*(r^*)}(v_5^*)
=
\operatorname{lev}_{T^*(r^*)}(v_6^*)
= 3.
\]
\end{example}

Let us recall now the concept of  labeled  trees.

\begin{defn}\label{rvjiu}
A labeled tree $T=T(l)$ is a tree $T$
together with a vertex labeling
\(
l:V(T)\to L,
\)
where \(L\) is a given set.
\end{defn}

In the following, we will consider only those vertex labelings $l:V(T)\to L$ for which $L=\mathbb{R}^+$.

\begin{defn}\label{poogf}
Let $T_1=T_1(l_1)$ and $T_2=T_2(l_2)$ be labeled trees.
A mapping $F:V(T_1)\to V(T_2)$  is called an isomorphism of 
$T_1(l_1)$ and $T_2(l_2)$ if 
$F$ is an isomorphism 
of the free trees
$T_1$ and $T_2$, and, moreover,
the equality
\begin{equation*}
l_2(F(v))=l_1(v)
\end{equation*}
holds for every $v\in V(T_1)$. The labeled  trees
 are isomorphic if there is an
isomorphism of these trees.
\end{defn}

Now we are ready to formulate the definition of isomorphism of labeled rooted trees.

\begin{defn}\label{saq}
Let $T_1=T_1(r_1,l_1)$ and $T_2=T_2(r_2,l_2)$ be labeled rooted trees.
A mapping $F:V(T_1)\to V(T_2)$  is called an isomorphism of 
$T_1(r_1,l_1)$ and $T_2(r_2,l_2)$ if 
$F$ is an isomorphism 
of the rooted trees
$T_1(r_1)$ and $T_2(r_2)$, and, 
simultaneously, $F$ is an isomorphism of the labeled trees $T_1(l_1)$ and $T_2(l_2)$. The labeled rooted trees are isomorphic if there exists an isomorphism of these trees.
\end{defn}

The next our goal is to remind the definition of representing trees of finite ultrametric spaces. To give a definition of representing trees, we will use the concepts of complete multipartite graphs and diametrical graphs.

\begin{defn}\label{edrtaa}
Let $k\geq 2$ be an integer number.
A non-empty graph $G$ is called complete $k$-partite if its vertices can be divided into disjoint non-empty sets $X_1,\ldots,X_k$ so that there are no edges joining the vertices which belong to the same $X_i$, and, moreover, any two vertices from different $X_i$ and $X_j$ are adjacent. In this case we write
\(
G=G[X_1,\ldots,X_k].
\)
\end{defn}

We shall say that $G$ is a \emph{complete multipartite graph} if there exists $k$ such that $G$ is complete $k$-partite.

\begin{defn}\label{okiu}
Let $(X,d)$ be a finite metric space with $|X|\geq 2$. The diametrical graph $G_X$ of $(X,d)$ is a graph such that $V(G)=X$ and, for all $u,v\in X$,
\[
\{u,v\}\in E(G_{X})
\iff
d(u,v)=\operatorname{diam}X.
\]
\end{defn}

The following theorem was proved in \cite{DDP2011pNUAA}.

\begin{thm}\label{kmjkusa}
Let $(X,d)$ be a finite ultrametric space with $|X|\geq 2$. Then the diametrical graph
\(
G_{X}
\)
is complete multipartite.
\end{thm}

Now we are ready to describe an algorithm for construction of representing trees $T_X=T_X(r_X,l_X)$ of finite ultrametric spaces $(X,d)$.

{\bf Construction of representing trees.} 
 Let $(X,d)$ be a finite ultrametric space.

 The tree $T_X(r_X,l_X)$ has the root 
 $r_X:=X$ with 
 \(
 l_X(r_X):=\operatorname{diam} X.
 \)

If $X$ is a single-point set, then the root $X$ is the unique vertex of $T_X$, and the construction is completed.

Let $|X|\ge 2$. Then, according to Theorem~\ref{kmjkusa}, the diametrical graph $G_X$ is complete multipartite,
\(
G_{X}=G_{X}[X_1,\ldots,X_k]
\)
with $k\ge 2$. In this case the sets $X_1,\ldots,X_k$ are the nodes of the first level of $T_X(r_X)$ and
\(
l_X(X_i):=\operatorname{diam}X_i
\)
for $i=1,\ldots,k.$
The nodes of the first level with the label $0$ are leaves, and those indicated by strictly positive labels are internal nodes of $T_X(r_X)$. If all nodes of the first level are leaves, then the representing tree $T_X(r_X,l_X)$ is constructed. 

If some nodes $X_i$ satisfy the inequality $l_X(X_i)>0$, then by repeating the above-described procedure with each $(X_i,d |_{X_i \times X_i})$ satisfying this inequality
 we obtain the nodes of the second level, etc. 
 
 Since the set $X$ is finite, all nodes on some level will be leaves, and the construction of $T_X(r_X,l_X)$ is completed.

\begin{rem}\label{bvy}
    The above described algorithm 
 was proposed in \cite{PD2014JMS}.
\end{rem}

\begin{thm}\label{hds11}
Let $T=T(r,l)$ be a finite labeled rooted tree with the root $r$
and the labeling $l:V(T)\to \mathbb{R}^{+}$.
Then the following two statements are equivalent.

\begin{enumerate}
\item[(i)]
For every $u\in V(T)$ we have $\delta_{T(r)}^{+}(u)\neq 1$, and
\[
(\delta_{T(r)}^{+}(u)=0)\iff (l(u)=0),
\]
and, in addition, the inequality
\begin{equation*}
l(z)<l(x)
\end{equation*}
holds whenever $x,z\in V(T)$ and $z$ is a direct successor of $x$.

\item[(ii)]
There exists a finite ultrametric space $(X,d)$ such that the representing
tree $T_X=T_X(r_X,l_X)$ and $T=T(r,l)$ are isomorphic as labeled rooted trees.
\end{enumerate}
\end{thm}

The proof of Theorem \ref{hds11}
can be found in \cite[Theorem 2.7]{DP2018pNUAA}.

\begin{thm}\label{rvqqbb}
Let $(X,d)$ and $(Y,\rho)$ be  finite ultrametric spaces. Then the representing trees
$T_X=T_X(r_X,
l_X)$ and $T_Y=T_Y(r_Y,l_Y)$ are isomorphic as labeled rooted trees if and only if
$(X,d)$ and $(Y,\rho)$ are isometric.
\end{thm}

Theorem \ref{rvqqbb} was formulated in \cite[Theorem 2.6]{DPT2017FPTA}.
A complete proof of Theorem \ref{rvqqbb} is given in \cite[Theorem 1.10]{DP2018pNUAA}.

In the next section of the paper we also use the following proposition.

\begin{prop}\label{effrt}
Let $(X,d)$ and $(Y,\rho)$ be finite ultrametric spaces with  representing trees $T_X=T_X(r_X,l_X)$ and $T_Y=T_Y(r_Y,l_Y)$, respectively. Then 
\(
T_X(r_X,l_X)
\) and \(
T_Y(r_Y,l_Y)
\)
are isomorphic as labeled rooted trees if and only if
\(
T_X(l_X)
\) and
\(T_Y(l_Y)
\)
are isomorphic as labeled trees.
\end{prop}

\begin{proof}
If $T_X(r_X,l_X)$ and $T_Y(r_Y,l_Y)$ are isomorphic, then $T_X(l_X)$ and $T_Y(l_Y)$ also are isomorphic by Definition~\ref{saq}.

Let $T_X(l_X)$ and $T_Y(l_Y)$ be isomorphic, and let
\(
F:V(T_X)\to V(T_Y)
\)
be an isomorphism of $T_X(l_X)$ and $T_Y(l_Y)$. We claim that $F:V(T_X)\to V(T_Y)$ also is an isomorphism of $T_X(r_X,l_X)$ and $T_Y(r_Y,l_Y)$.

Definition \ref{wf} and Definition \ref{saq} imply that the above claim is valid if and only if the equality
\begin{equation}
F(r_X)=r_Y
\label{eq:star1}
\end{equation}
holds.

Let us prove equality \eqref{eq:star1}. Since $F$ is an isomorphism of $T_X(l_X)$ and $T_Y(l_Y)$, we have the equality
\begin{equation}
l_Y(F(v_X))=l_X(v_X)
\label{eq:star2}
\end{equation}
for each $v_X\in V(T_X)$.
Theorem \ref{hds11} implies the inequalities
\begin{equation}
l_X(v_X)<l_X(r_X)
\label{eq:star3}
\end{equation}
and
\begin{equation}
l_Y(v_Y)< l_Y(r_Y)
\label{eq:star4}
\end{equation}
for all $v_X\in V(T_X)\setminus\{r_X\}$ and $v_Y\in V(T_Y)\setminus\{r_Y\}$.
Since $F:V(T_X)\to V(T_Y)$ is bijective, equality \eqref{eq:star2} and inequalities \eqref{eq:star3}, \eqref{eq:star4} imply
\begin{equation}
l_Y(r_Y)=l_X(r_X).
\label{eq:star5}
\end{equation}
If \eqref{eq:star1} does not hold, then there is $u_Y\in V(T_Y)\setminus\{r_Y\}$ such that
\begin{equation}
F(r_X)=u_Y.
\label{eq:star6}
\end{equation}
Now using \eqref{eq:star2} with $v_X=r_X$ and \eqref{eq:star6}, we obtain
\begin{equation*}
l_X(r_X)=l_Y(F(r_X))=l_Y(u_Y).
\end{equation*}
Since $u_Y\in V(T_Y)\setminus\{r_Y\}$, the strict inequality
\(
l_Y(u_Y)<l_Y(r_Y)
\)
holds. The last inequality and equalities \eqref{eq:star5} imply
\begin{equation*}
l_X(r_X)<l_Y(r_Y),
\end{equation*}
contrary to \eqref{eq:star5}. 

Equality \eqref{eq:star1} follows.
\end{proof}

\section{Auxiliary lemmas}\label{seq3}

Let us start with lemmas about leaves of rooted trees.

\begin{lem}\label{pit}
Let $T=T(r)$ be a rooted tree. Then the following statements are equivalent.
\begin{enumerate}
\item
The vertex set of $T$ contains the unique vertex $r$, 
\begin{equation}
    \label{y1}
V(T)=\{r\}. 
\end{equation}

\item
The root $r$ is a leaf of $T(r)$,
\begin{equation}
    \label{y2}
r\in L(T(r)).
\end{equation}
\end{enumerate}
\end{lem}

\begin{proof}
Let equality \eqref{y1} hold. 
Then formula \eqref{mmks} implies the equality
\(
\delta_T(r)=0 
\)
and, consequently,
we have
\(
\delta_{T(r)}^+(r)=0
\)
by formula \eqref{zbh1}.
The equality $\delta_{T(r)}^+(r)=0$ and Definition~\ref{scm} give us relation 
\eqref{y2}. 

Thus the implication $(i)\Rightarrow(ii)$ is valid.

Suppose now that $(ii)$ is true but \eqref{y1} does not hold. Then the inequality
\(
|V(T)|\ge 2
\)
holds and, consequently, we can find a vertex
\(
u_0\in V(T)\setminus\{r\}.
\)
Let $P_{u_0,r}$ be the path joining $u_0$ and $r$ in $T$.

It follows directly from the definition of the paths (see formula \eqref{gasl}) that 
there is $v_{0}\in V(P_{u_{0},r})$ such that
\(
\{v_{0},r\}\in E(P_{u_{0},r}).
\)
Since $P_{u_{0},r}$ is a subgraph of $T$,
we also have the relation
\begin{equation}
    \label{bb1}
\{v_{0},r\}\in E(T). 
\end{equation}
Using \eqref{bb1}, we obtain
\begin{equation*}
\delta_{T}(r)
\geq 1
\end{equation*}
and, consequently, $\delta_{T(r)}^+(r)
\geq 1$ by formula \eqref{zbh1}.
Hence, by Definition~\ref{scm}, the root $r$ is not a leaf of $T(r)$, contrary to statement $(ii)$.

 Thus the implication
\(
(ii)\Rightarrow(i)
\)
also is valid.
\end{proof}

The following lemma apparently known.

\begin{lem}\label{rgj}
Let $n \ge 0$ be an integer number and let $T=T(r)$ be a finite perfect binary tree such that
\begin{equation*}
\operatorname{lev}_{T(r)}(v)=n
\end{equation*}
for each leaf $v$ of $T(r)$.
Then the number of  leaves of $T(r)$ is $2^n$,
\begin{equation}
    \label{ew3}
|L(T(r))|=2^n.
\end{equation}
\end{lem}

\begin{proof}
To prove \eqref{ew3} we note that equality
 \eqref{ew3}  holds  by Lemma \ref{pit} if $n=0$.

For the case $n\geq 1$ 
we can use Theorem \ref{gjl} and prove equality \eqref{ew3} by induction.
\end{proof}

The following two lemmas describe some interconnections between the points of finite ultrametric spaces and the leaves of their representing trees.

\begin{lem} \label{UT2Lemma1}
Let $(X,d)$ be a finite ultrametric space, let $p$ and $q$ be distinct points of $X$, and $T_X = T_X(r_X, l_X)$ be the representing tree of $(X,d)$. Then the equality
\begin{equation*}
d(p,q) = \max_{v \in V(P_{\{p\},\{q\}})} l_X(v)
\end{equation*}
holds, where $P_{\{p\},\{q\}}$ is the path joining the leaves $\{p\}$ and $\{q\}$ in $T_X(r_X)$.
\end{lem}

The proof of Lemma~\ref{UT2Lemma1} is completly similar to the proof of Lemma~3.2 from \cite{PD2014JMS}, so we omit it here.

Let $(X,d)$ be a metric space. Then, for every $p \in X$, we define a set $D_p(X)$ as
\begin{equation}
    \label{mkkggd}
D_p(X) := \{\, d(p,x) : x \in X \,\}.
\end{equation}

\begin{lem} \label{UT2Lemma2}
 Let $(X,d)$ be a finite ultrametric space with $|X|\geq 2$ and let $T_X = T_X(r_X,l_X)$ be the representing tree of $(X,d)$. Then the equality
\begin{equation}\label{iuhds}
D_p(X)=\{l_X(v):v\in V(P_{\{p\},r_X})\}
\end{equation}
holds for each $p \in X$.
\end{lem}

\begin{proof}
Let us consider an arbitrary \(p\in X\).
The inclusion
\begin{equation*}
D_p(X)\subseteq
\left\{l_{X}(v):\, v\in V(P_{\{p\},r_X})\right\}
\end{equation*}
follows from Lemma \ref{UT2Lemma1}.

Let us prove the converse inclusion,
\begin{equation}
    \label{qmn349}
D_p(X)\supseteq
\left\{l_{X}(v):\, v\in V(P_{\{p\},r_X})\right\}.
\end{equation}

Proposition \ref{xxtt} implies the equality
\begin{equation}
    \label{444}
\max\{d(p,a):a\in A\}=\operatorname{diam}A
\end{equation}
for each $A\subseteq X$ with $p\in A$. By definition of $T_X(r_X,l_X)$, every $u\in V(T_X)$ is a subset of $X$ and the equality 
\begin{equation}\label{rar}
    l_X(u)=\operatorname{diam} u
\end{equation}
holds.
Since each vertex \(v\) of \(P_{\{p\},r_X}\) is a subset of \(X\) satisfying the relation \(p\in v\),
equalities \eqref{rar} and \eqref{444} imply
\begin{equation*}
l_X(v)\in D_p(X)
\end{equation*}
for every $v \in V(P_{\{p\},r_X})$.
Thus \eqref{qmn349} holds. 

Equality \eqref{iuhds} follows.
\end{proof}

The next lemma directly follows from Definition \ref{zcgh572} and equality \eqref{mkkggd}.

\begin{lem}\label{xute}
 Let $(X,d)$ be a metric space.
Then the equality
\begin{equation*}
    C(X)=\bigcap_{p\in X} D_p(X)
\end{equation*}
holds.
\end{lem}

The following two lemmas are reformulations of Proposition~3.4 and Corollary~3.5 from \cite{DovgosheyRovenska2026}.

\begin{lem}\label{suul}
Let $(X,d)$ be a finite ultrametric space with $|X|\geq 2$ and let the diametrical graph $G(X)$ be complete $k$-partite with the parts $X_1,\ldots,X_k$.
Then the equality
\begin{equation*}
C(X)=\{\operatorname{diam} X\}\cup\left(\bigcap_{i=1}^{k} C(X_i)\right)
\end{equation*}
holds, where $C(X_1),\dots,C(X_k)$ are the centers of distances of the ultrametric spaces $(X_1,d|_{X_1 \times X_1})$, $\dots$, $(X_k,d|_{X_k \times X_k})$.
\end{lem}

\begin{lem} \label{kvcr}
Let $(X,d)$ be a finite ultrametric space with $|X|\geq 2$, let the diametrical graph $G_X$ be complete $k$-partite with the parts $X_1,\ldots,X_k$,
and let $X_{k_0}$ be a part of $G_X$ such that
\[
|X_{k_0}|=\min\limits_{1\leq j\leq k}|X_j|.
\]
Then the inequality
\[
|C(X)| \leq 1+|C(X_{k_0})|
\]
holds, where 
$C(X_{k_0})$ is the center of distances of the ultrametric space 
$(X_{k_0}, d|_{X_{k_0}\times X_{k_0}})$.
\end{lem}

The next lemma shows, in particular, that the representing trees of ${\bf MCD}$-spaces are perfect binary  trees.

\begin{lem} \label{UT2Lemma3}
Let $(X,d)$ be a ${\bf MCD}$-space and let $n \geq 1$ be an integer number such that 
\begin{equation}
    \label{fl}
|X|=2^n
\quad\text{and}\quad
|C(X)|=n+1.
\end{equation}
\begin{enumerate}
\item[(i)]  The representing tree
\(
T_{ X}=T_{ X}(r_{ X},l_{ X})
\)
of $( X,d)$ is a binary tree.

\item[(ii)]  The equality
\begin{equation*}
\operatorname{lev}_{T_X(r_X)}(v)=n
\end{equation*}
holds for each leaf $v$ of $T_{ X}(r_X)$.

\item[(iii)] The equalities
\begin{equation}
    \label{o1}
D(X)=C(X)=\{\,l_X(v):v\in V(T_X)\,\}=D_p(X)
\end{equation}
hold for each \(p\in X\),
where $D( X)$ is the distance set of $( X,d)$ and
$D_p( X)$ is defined by formula~\eqref{mkkggd}.
\end{enumerate}
\end{lem}

\begin{proof}
Let $n=1$. Then statements $(i)$--$(iii)$ are evidently valid for each ultrametric space $(X,d)$ that satisfies \eqref{fl}.

Let $n_0 \geq 2$ be an integer number and let statements $(i)$--$(iii)$ hold for all ultrametric spaces $(X,d)$ satisfying \eqref{fl} with $1\leq n<n_0$.

Let us consider an arbitrary ultrametric space $(X,d)$ satisfying \eqref{fl} with $n=n_0$. We must prove that $(i)$--$(iii)$ are valid for this space.

$(i)$. Let us prove the validity of statement $(i)$ for $n=n_0$. To do it we consider the diametrical graph $G_X$ of $(X,d)$. First of all we will show that $G_X$ is a complete bipartite graph.

By Theorem \ref{kmjkusa} the diametrical graph $G_X$ is a complete $k$-partite with $k\ge 2$.  Let $X_1,\dots,X_k$ be the parts of $G_X$ and let $k_0\in\{1,\dots,k\}$ satisfy the equality
\begin{equation}\label{eq:min-part}
|X_{k_0}|=\min_{1\le j\le k}|X_j|.
\end{equation}
Then Lemma \ref{kvcr} implies the inequality
\begin{equation}\label{eq:center-1}
|C(X)|\le 1+|C(X_{k_0})|,
\end{equation}
where $C(X_{k_0})$ is the center of distances of $\bigl(X_{k_0},d|_{X_{k_0}\times X_{k_0}}\bigr)$. It follows from Theorem~\ref{ter1-1} that
\begin{equation*}
|C(X_{k_0})|\le 1+\left\lfloor \log_2 |X_{k_0}| \right\rfloor,
\end{equation*}
and, consequently, we have
\begin{equation}\label{eq:center-3}
|C(X)|\le 2+\left\lfloor \log_2 |X_{k_0}| \right\rfloor
\end{equation}
by \eqref{eq:center-1}.
Definition \ref{edrtaa} and the first equality in \eqref{fl} with $n=n_0$ give us
\begin{equation}\label{eq:W-size}
|X|=2^{n_0}=\sum_{j=1}^{k}|X_j|.
\end{equation}
Hence if $G_X$ is a complete $k$-partite graph with $k\geq 3$, then \eqref{eq:min-part} and \eqref{eq:W-size} imply the strict inequality
\(
|X_{k_0}|<2^{n_0-1},
\)
and, consequently,  the inequality
\begin{equation}\label{eq:log-bound}
\left\lfloor \log_2 |X_{k_0}| \right\rfloor \le n_0-2.
\end{equation}
It follows from \eqref{eq:center-3} and \eqref{eq:log-bound} that
\begin{equation}\label{eq:center-bound}
|C(X)|\le n_0.
\end{equation}
The second equality in \eqref{fl} with $n=n_0$ gives us the equality
\begin{equation*}
|C(X)|=n_0+1,
\end{equation*}
that contradicts \eqref{eq:center-bound}. 

Thus $G_X$ is a complete bipartite graph. Now using induction hypothesis for statement $(i)$, Definition \ref{xxetr} and the algorithm for construction of representing trees we see that $T_X(r_X)$ is a binary rooted tree. Thus statement $(i)$ holds with $n=n_0$.

$(ii)$. Let us prove statement $(ii)$ for $n=n_0$.

It was shown in the proof of statement $(i)$ that the diametrical
graph $G_X$ is complete bipartite. Let $X_1$ and $X_2$ be the parts of
$G_X$, and let $(X_1,d|_{X_1\times X_1})$, $(X_2,d|_{X_2\times X_2})$ be the corresponding subspaces of $(X,d)$.
We claim that the equalities
\begin{equation}\label{eq:e1}
|X_1|=|X_2|=2^{n_0-1}
\end{equation}
and
\begin{equation}\label{eq:e2}
|C(X_1)|=|C(X_2)|=n_0
\end{equation}
hold.

Let us prove \eqref{eq:e1}. If \eqref{eq:e1} is false, then, 
the equalities $|X|=2^{n_0}$ and $|X|=|X_1|+|X_2|$ imply
the strict inequality
\begin{equation*}
\min\{|X_1|,|X_2|\}<2^{n_0-1},
\end{equation*}
that, as in the proof of statement $(i)$, gives us a contradiction with the equality $|C(X)|=n_0+1$. 
Thus \eqref{eq:e1} holds.

Let us prove \eqref{eq:e2}. 

Theorem~\ref{ter1-1} and equalities \eqref{eq:e1} imply
\begin{equation}\label{eq:e4}
|C(X_i)|\le 1+\lfloor \log_2 |X_i| \rfloor = 1+n_0-1=n_0
\end{equation}
for $i=1,2$.
The second equality in \eqref{fl}  with \(n=n_0\) and Lemma~\ref{kvcr} imply
\begin{equation}\label{eq:e5}
n_0+1=|C(X)|\le 1+\min\{|C(X_1)|,|C(X_2)|\}.
\end{equation}
It follows from \eqref{eq:e4} and \eqref{eq:e5} that
\begin{equation}\label{eq:e6}
n_0=\min\{|C(X_1)|,|C(X_2)|\}
\end{equation}
holds.

Equality \eqref{eq:e2} follows from \eqref{eq:e4} and \eqref{eq:e6}.

Let $v$ be an arbitrary leaf of $T_X(r_X)$. To prove statement $(ii)$ with $n=n_0$, we must prove the equality
\begin{equation}
    \label{oo1}
\operatorname{lev}_{T_X(r_X)}(v)=n_0.
\end{equation}

The definition of representing trees implies that the equality
\begin{equation}
    \label{oo2}
    v=\{p\}
\end{equation}
holds for some point $p\in X$.

Let $p \in X$ satisfy equality  \eqref{oo2}. Since $\{X_1,X_2\}$ is a partition of the set $X$, we have either 
\(
p\in X_1\) or \(p\in X_2.
\)
Assume, 
without loss of generality, that $p\in X_1$. Then, using the definition
of $T_{X_1}(r_{X_1},l_{X_1})$, we obtain
\begin{equation}
    \label{oo3}
\{p\}\in L(T_{X_1}(r_{X_1})).
\end{equation}
Relation \eqref{oo3}, the equalities
\begin{equation*}
    |X_1|=2^{n_0-1}, \quad |C(X_1)|=n_0,
\end{equation*}
(see equalities \eqref{eq:e1}, \eqref{eq:e2}) and the induction hypothesis imply the equality
\begin{equation}
    \label{oo4}
\operatorname{lev}_{T_{X_1}(r_{X_1})}(\{p\})=n_0-1.
\end{equation}
It follows from the definition of
\(
T_X=T_X(r_X,l_X)
\)
that
\(
\deg_{T_X}(r_X)=2
\)
and that $r_{X_1},$ $r_{X_2}$ are the direct successors of $r_X$.
Consequently, the equality
\begin{equation}
    \label{oo5}
|E(P_{\{p\},r_X})|
=
1+|E(P_{\{p\},r_{X_1}})|
\end{equation}
holds, where $P_{\{p\},r_X}\subseteq T_X$ and $P_{\{p\},r_{X_1}}\subseteq T_X$ are the paths
joining $\{p\}$ with $r_X$ and, respectively, with $r_{X_1}$.

Now equality \eqref{oo1} follows from equalities \eqref{lev}, \eqref{oo2}, \eqref{oo4}, and \eqref{oo5}.

$(iii)$.
Let us prove statement $(iii)$. Let $p_1$ and $p_2$ be two distinct points of $(X,d)$.
We claim that the equality
\begin{equation}
    \label{av1}
D_{p_1}(X)=D_{p_2}(X) 
\end{equation}
holds, where 
 \(D_{p_1}(X)\) and \(D_{p_2}(X)\) are defined
by formula \eqref{mkkggd} with \(p=p_1\) and \(p=p_2\), respectively.
Indeed, suppose contrary that
\begin{equation}
    \label{vv1}
D_{p_1}(X)\ne D_{p_2}(X).
\end{equation}
Lemma \ref{xute} gives us the equality
\begin{equation}
    \label{v2}
C(X)=\bigcap_{p\in X}D_p(X) 
\end{equation}
and, 
moreover, we have the equalities
\begin{equation}
    \label{v3}
D_{p_1}(X)
=
\left\{
l_X(v) : v\in V(P_{\{p_1\},r_X})
\right\}, 
\quad
D_{p_2}(X)
=
\left\{
l_X(v) : v\in V(P_{\{p_2\},r_X})
\right\}
\end{equation}
by Lemma \ref{UT2Lemma2}, and the equalities
\begin{equation}
    \label{v5}
\operatorname{lev}_{T_X(r_X)}(\{p_1\})=\operatorname{lev}_{T_X(r_X)}(\{p_2\})=n
\end{equation}
by statement $(ii)$.
Using the second equality in \eqref{fl}
and equality \eqref{v2}, we obtain
\begin{equation}
    \label{v6}
|n+1|
=
|C(X)|
=
\big|\bigcap_{p\in X}D_p(X)\big|\leq
|D_{p_1}(X)\cap D_{p_2}(X)|. 
\end{equation}
Equalities \eqref{lev}, \eqref{v3}--\eqref{v5} imply 
\begin{equation}
\label{v7}
|D_{p_1}(X)|\leq |V(P_{\{p_1\},r_X}|=1+|E(P_{\{p_1\},r_X}|=1+n.
\end{equation}
If \eqref{vv1} holds then we also have the strict inequality
\begin{equation*}
|D_{p_1}(X)\cap D_{p_2}(X)|
<
\max\{|D_{p_1}(X)|,\ |D_{p_2}(X)|\},
\end{equation*}
that, together with 
\eqref{v6}--\eqref{v7}, gives the false inequality
\[
n+1<n+1.
\]
Thus equality \eqref{av1} holds for all \(p_1,p_2\in X\).

Now equalities \eqref{o1} follow from equalities \eqref{av1}, \eqref{v2} and
the equality
\[
D(X)=\bigcup_{p\in X}D_p(X).
\]
\end{proof}

Analyzing the proof of Lemma \ref{UT2Lemma3} we obtain the
following.

\begin{lem}
    \label{lemlem}
Let $(X,d)$ be a ${\bf MCD}$-space, let $n\geq 1$ be an
integer number satisfying the equalities
\begin{equation}
    \label{ddaad}
|X|=2^n \quad\text{and}\quad |C(X)|=n+1,
\end{equation}
and let $G(X)$ be the diametrical graph of $(X,d)$. Then
$G(X)$ is complete bipartite with parts $X_1,$ $X_2$ such that
\begin{equation}
    \label{ddad2}
|X_1|=|X_2|=2^{n-1}
\quad\text{and}\quad
|C(X_1)|=|C(X_2)|=n,
\end{equation}
where \(C(X_1)\) and \(C(X_2)\) are the centers of distances of ultrametric spaces
\(
(X_1,d|_{X_1\times X_1})\) and \(
(X_2,d|_{X_2\times X_2}),
\)
respectively.
\end{lem}

\begin{proof}
It was shown in the proof of statement $(i)$  of Lemma \ref{UT2Lemma3} that
\(G_X\) is complete bipartite if \eqref{ddaad} holds with $n\geq2$.
Moreover, using equalities  \eqref{y1} and \eqref{eq:e1}, we obtain that equalities
\eqref{ddaad} imply \eqref{ddad2} if \(n\geq 2\).

To complete the proof it suffices to note that the conclusion of Lemma \ref{lemlem}
directly follows from Definitions \ref{edrtaa}, \ref{okiu} and Theorem \ref{kmjkusa} if \(n=1\).
\end{proof}

\section{Isometries of {\bf MCD}-spaces and isomorphisms of their representing trees }\label{seq4}

The first our theorem describes the structure of representing trees
of $\bf MCD$-spaces.

\begin{thm}\label{mmre}
Let $T=T(r,l)$ be a finite labeled rooted tree and let $L(T(r))$ be the set of leaves of $T(r)$.
Then the following two statements are equivalent.

\begin{enumerate}
\item[(i)]
The rooted tree $T(r)$ is a perfect binary tree, and the equality
\begin{equation}
    \label{ol1}
l(v)=0
\end{equation}
holds for each $v \in L(T(r))$, and the equivalence
\begin{equation}
    \label{td}
    \left( \operatorname{lev}_{T(r)}(w)=\operatorname{lev}_{T(r)}(x)\right) \iff     \left( l(w)=l(x)\right) 
\end{equation}
is valid for all $w,x\in V(T)$, and, moreover,
the inequality
\begin{equation}\label{fvhjj-1}
l(z)<l(y)
\end{equation}
holds whenever $y,z\in V(T)$ and $z$ is a direct successor of $y$.

\item[(ii)]
There exists a $ {\bf MCD}$-space $(X,d)$ for which the representing tree
\(
T_X=T_X( r_X,l_X)
\)
and
\(
T=T(r,l)
\)
are isomorphic as labeled rooted trees.
\end{enumerate}
\end{thm}

\begin{proof} 
 $(i)\Rightarrow(ii)$.
Let $(i)$ hold. Then by Definition \ref{xxetr} there is
an integer number $n\geq 0$ such that
\begin{equation}
\label{e1}
\operatorname{lev}_{T(r)}(v)=n
\end{equation}
for every $v\in L(T(r))$.

If $n=0$, then the equalities
\begin{equation}\label{j1}
V(T)=L(T(r))=\{r\}
\end{equation}
hold by Lemma \ref{pit}. Equalities \eqref{ol1} and \eqref{j1} imply 
\begin{equation*}
l(r)=0.
\end{equation*}
Moreover, for all single-point ultrametric spaces $(X,d)$, the definition of representing trees gives us the equalities
\begin{equation*}
r_X=X, \quad
V(T_X)=L(T_X(r_X))=\{X\}
\end{equation*}
and
\begin{equation}\label{j4}
l_X(r_X)=\operatorname{diam} X=0,
\end{equation}
where $\{X\}$ is a singleton with the unique point $X$.
Since each single-point $(X,d)$ belongs to ${\bf MCD}$, equalities \eqref{j1}--\eqref{j4} imply the validity of the implication
\(
(i)\Rightarrow(ii) 
\) for $n=0$.

Let us consider the case when 
\begin{equation}
    \label{ccc}
    n\geq 1.
\end{equation}
Statement $(i)$ and Theorem~\ref{hds11} imply the existence of finite ultrametric space $(X,d)$ such that
\(
T_X=T_X(r_X,l_X)
\)
and
\(
T=T(r,l)
\)
are isomorphic as labeled rooted trees. We must show that $(X,d)$ is a ${\bf MCD}$-space.

Let us do it. 
 Let  \(L(T_X(r_X))\) be the set of leaves of \( T_X(r_X)\). Since $T_X(r_X,l_X)$ and $T(r,l)$ are isomorphic labeled rooted trees, the equality
\begin{equation}\label{d2}
|L(T_X(r_X))|=|L(T(r))|
\end{equation}
holds.
By \eqref{e1} we have $\operatorname{lev}_{T(r)}(v)=n$ for each  $v \in L(T(r))$ and, consequently,
 we also have the equality
\begin{equation*}
|L\!\left(T(r)\right)|=2^n
\end{equation*}
by Lemma \ref{rgj}.
It follows directly from the definition of representing trees that
\begin{equation*}
L\!\left(T_X(r_X)\right)=\{\{p\}:\, p\in X\}.
\end{equation*}
Hence the equality 
\begin{equation}\label{d4}
|L\!\left(T_X(r_X)\right)|=|X|
\end{equation}
holds.
Equalities \eqref{d2}--\eqref{d4} imply the  equality
\begin{equation}
    \label{ol2}
|X|=2^n.
\end{equation}

Let us prove the equality
\begin{equation}
\label{ol3}
|C(X)|=n+1.
\end{equation}

It follows from \eqref{ccc} and \eqref{ol2}
that the inequality \(|X|\ge2\) holds and,
consequently, we can use
Lemma~\ref{UT2Lemma2}. This
lemma gives us the equality
\begin{equation}
    \label{c1}
D_p(X)=\{\,l_X(u):u\in V(P_{\{p\},r_X})\,\}
\end{equation}
for each \(p\in X\), where
\(
D_p(X):=\{\,d(p,x):x\in X\,\}
\)
and \(P_{\{p\},r_X}\) is the path joining the leaf \(\{p\}\) with the root
\(r_X\) in \(T_X(r_X,l_X)\).
By Lemma~\ref{xute} we have the equality
\begin{equation}
    \label{c2}
C(X)=\bigcap_{p\in X} D_p(X).
\end{equation}
It follows from \eqref{c1} and \eqref{c2} that
\begin{equation}\label{c3}
C(X)=\bigcap_{p\in X} \{l_X(u): u\in V(P_{\{p\},r_X})\}.
\end{equation}
Since \eqref{e1} holds for every $v\in L(T(r))$, equivalence \eqref{td} gives us the equality
\begin{equation}
    \label{c4}
\{l(u):u\in V(P_{v_1,r})\}
=
\{l(u):u\in V(P_{v_2,r})\}
\end{equation}
for any two leaves \(v_1,v_2\) of \(T(r,l)\). 
Equality  \eqref{c4} and isomorphism of \(T(r,l)\) and \(T_X(r_{X},l_X)\) imply  
\[
\{l_X(u):u\in V(P_{\{p_1\},r_X})\}
=
\{l_X(u):u\in V(P_{\{p_2\},r_X})\}
\]
for any two points \(p_1,p_2\in X\).
Consequently we can rewrite \eqref{c3} in the form
\begin{equation}
    \label{x1}
C(X)=\{l_X(u):u\in V(P_{\{p\},r_X})\},
\end{equation}
where \(p\) is an arbitrary point of \(X\).

Once again, since 
\(T(r,l)\) and \(T_X(r_X,l_X)\) are isomorphic, equality \eqref{x1} implies
\[
C(X)=\{\,l(u): u\in V(P_{v,r})\,\}
\]
for every leaf \(v\) of \(T(r,l)\).
Hence, the equality
\begin{equation}\label{wsq}
|C(X)|=\bigl|\{l(u):u\in V(P_{v,r})\}\bigr|
\end{equation}
holds for every leaf \(v\) of \(T(r,l)\). Let us consider an arbitrary leaf $v_0$ of $T(r,l)$.
It follows from \eqref{fvhjj-1} that
\[
l(u_1)\neq l(u_2)
\]
holds whenever \(u_1,u_2\) are distinct vertices of the path
\(P_{v_0,r}\). The relation $(X,d)\in {\bf MCD}$ implies the existence of integer $m\geq 0$ such that
\begin{equation}\label{ggd}
|C(X)|=|V(P_{v_0,r})|
\end{equation}
holds by \eqref{wsq}.
Since the equality
\(
|V(P)|=1+|E(P)|
\)
holds for each path $P$, \eqref{lev} and \eqref{ggd} imply
\[
|C(X)|=1+\operatorname{lev}_{T(r)}(v_0).
\]
The last equality and equality \eqref{e1} give
us \eqref{ol3}.

The relation $(X,d)\in {\bf MCD}$ follows from \eqref{ol2} and \eqref{ol3} by definition of ${\bf MCD}$-spaces.

Thus statement $(i)$ implies statement $(ii)$.

$(ii)\Rightarrow(i).$
Let \((X,d)\) be a ${\bf MCD}$-space such that the representing
tree
\(
T_X=T_X(r_X,l_X)
\)
is isomorphic to
\(
T=T(r,l).
\)

Since $(X,d)$ is a ${\bf MCD}$-space, there exists an integer $m \geq 0$ such that
\begin{equation}
    \label{e2}
|X| = 2^m
\quad\text{and}\quad
|C(X)| = m + 1.
\end{equation}
If  $m=0$ holds, then arguing as in the proof of the validity
$(i)\Rightarrow(ii)$ for $n=0$, we can show that the implication
$(ii)\Rightarrow(i)$ is true.

Let us consider the case when 
\begin{equation}
    \label{e3}
    m\geq 1.
\end{equation}

The first equality in \eqref{e2} and inequality \eqref{e3} imply
 $|X|\geq 2$ and, consequently, we can use Lemma \ref{UT2Lemma3}.

By Lemma \ref{UT2Lemma3}, $T_X(r_X)$ is a perfect binary tree. Hence $T(r)$ also is a perfect binary rooted tree, since $T_X(r_X,l_X)$ and $T(r,l)$ are isomorphic as labeled rooted trees. By definition of representing trees we have
\begin{equation}\label{k2}
    l_X(\{p\})=0
\end{equation}
for each $p\in X$. Since 
\begin{equation*}
    L(T_X(r_X))=\{\{p\}:\,p\in X\}
\end{equation*}
holds, \eqref{k2} implies \eqref{ol1} for every $v\in L(T(r))$.

Let $x$ and $z$ be  vertices of $T(r)$  and let \(z\) be a direct successor of
\(x\). Then inequality \eqref{fvhjj-1} follows from Theorem~\ref{hds11}, because $T(r,l)$ and $T_X(r_X,l_X)$ are isomorphic as labeled rooted trees.

Thus, to complete the proof of  validity of the implication
\(
(ii)\Rightarrow(i)
\)
it suffices to show that equivalence \eqref{td}
is valid for all \(w,u\in V(T)\).

Since $T(r,l)$ and $T_{X}(r_{X},l_{X})$ are isomorphic labelled rooted trees, equivalence \eqref{td} is valid for all $u, w \in V(T)$ if and only if 
\begin{equation}
    \label{sss}
    \left( \operatorname{lev}_{T_X(r_X)}(w_1)=\operatorname{lev}_{T_X(r_X)}(w_2)\right) \iff     \left( l_X(w_1)=l_X(w_2)\right) 
\end{equation}
is valid for all \(w_1,w_2\in V(T_X)\).

Let us consider two arbitrary vertices $w_1$ and $w_2$ of ${T}_{X}$. Then, using inequality \eqref{e3} and the definition of representing trees, we can find
points $p_{1},p_{2}\in X$ such that
\(
w_1\in V(P_{\{p_1\},r_X})\)
and
\(
w_2\in V(P_{\{p_2\},r_X})\).

Statement $(iii)$ of Lemma~\ref{UT2Lemma3} implies 
the equality

\begin{equation}
    \label{k3}
\{d(p_{1},x):x\in X\}
=
\{d(p_{2},x):x\in X\}.
\end{equation}
By Lemma~\ref{UT2Lemma2} we have the equality
\[
\{d(p,x):x\in X\}
=
\{l_{X}(u):u\in  V(P_{\{p\},r_{X}})\}
\]
for each \(p\in X\).
The last equality and \eqref{k3} imply
\begin{equation*}
\{l_{X}(u):u\in V(P_{\{p_1\},r_X})\}
=
\{l_{X}(u):u\in V(P_{\{p_2\},r_X})\}.
\end{equation*}
By statement $(ii)$ of Lemma \ref{UT2Lemma3} we also have
\begin{equation}
    \label{ee1}
\operatorname{lev}_{T_X(r_X)}(\{p_1\})=\operatorname{lev}_{T_X(r_X)}(\{p_2\})=m.
\end{equation}

The vertices of the paths \(P_{\{p_1\},r_X}\) and \(P_{\{p_2\},r_X}\) can be numbered such that
\begin{equation*}
V(P_{\{p_i\},r_X})
 =\{v_0^i,\ldots,v_m^i\},
\quad
E(P_{\{p_i\},r_X})
 =\{\{v_0^i,v_1^i\},\ldots,\{v_{m-1}^i,v_m^i\}\},
\end{equation*}
where $i=1,2$, and
\begin{equation}
    \label{rer4}
v_0^1=\{p_1\},\quad
v_0^2=\{p_2\},\quad
v_m^1=v_m^2=r_X
\end{equation}
(see formulas \eqref{gasl}). 
Now, using \eqref{lev}--\eqref{p1} and \eqref{ee1}--\eqref{rer4}, we obtain the equality
\begin{equation}\label{ee4}
\operatorname{lev}_{T_X(r_X)}(v_k^i)=m-k
\end{equation}
for all $k\in\{0,\ldots,m\}$ and $ i\in \{1,2\}.$
Moreover, the definition of representing trees gives us 
the equalities
\begin{equation*}
l_X(v_0^1)=l_X(v_0^2)=0
\end{equation*}
and
the strict inequality
\begin{equation}\label{ee6}
l_X(v_{k+1}^i)>l_X(v_k^i)
\end{equation}
for all $k\in\{0,\ldots,m\}$ and $ i\in \{1,2\}.$

The validity of equivalence \eqref{sss} follows from \eqref{ee4}--\eqref{ee6}.

\end{proof}

The following theorem gives us a solution of Problem \ref{qwed}.

\begin{thm}\label{m0}
Let $(X,d)$ and $(Y,\rho)$ be $\bf MCD$-spaces.
Then the following statements are equivalent.
\begin{enumerate}
\item $(X,d)$ and $(Y,\rho)$ are isometric.

\item The representing trees
\(
T_X=T_X(r_X,l_X)
\) and
\(T_Y=T_Y(r_Y,l_Y)
\)
are isomorphic as labeled rooted trees.

\item The trees $T_X(l_X)$ and $T_Y(l_Y)$ are isomorphic as labeled trees.

\item The equality
\begin{equation}
    \label{m2}
C(X)=C(Y)
\end{equation}
holds.

\item The equality
\begin{equation*}
D(X)=D(Y)
\end{equation*}
holds.
\end{enumerate}
\end{thm}

\begin{proof} The equivalence 
$(i)\iff (ii)$ is valid by Theorem~\ref{rvqqbb}.

The validity of the equivalence $(ii) \iff (iii)$ follows from Proposition~\ref{effrt}.

Statement~$(iii)$ of Lemma~\ref{UT2Lemma3} implies the validity of 
\(
(iv)\iff(v).
\)
 It follows from Definition \ref{zcgh572} that the centers of distances
of isometric metric spaces are the same. Hence the implication
$(i) \Rightarrow (iv)$ is valid also. 

Thus to complete the proof  it suffices
to prove the validity of the implication
$(iv) \Rightarrow (i).$

Let statement $(iv)$ hold.
Then equality \eqref{m2} implies the existence of an integer $n \ge 0$ such that
\begin{equation}
    \label{v1-s}
|C(X)| = |C(Y)| = n+1. 
\end{equation}

We will prove that $(X,d)$ and $(Y,\rho)$ are isometric by induction.
If $n=0$ then we have 
\begin{equation*}
|X| = |Y| = 1,
\end{equation*}
since the definition of ${\bf MCD}$-spaces and \eqref{v1-s} imply 
\begin{equation*}
|X| = |Y| = 2^n=1.
\end{equation*}
Since all
single-point metric spaces are isometric, statement $(i)$ is valid for  $n=0.$

Let us consider the case $n\geq 1$. Let $n_0 \geq 1$ be an integer number and let statement $(i)$ hold for all ${\bf MCD}$-spaces $(X,d)$, $(Y,\rho)$ satisfying \eqref{v1-s} with $n<n_0$. 

By Lemma \ref{lemlem} the diametrical graphs
$G_X$ and $G_Y$ are complete bipartite. Let $X_1,$ $X_2$ be the parts of
$G_X$, and $Y_1,$ $Y_2$ be the parts of $G_Y$. By Lemma \ref{lemlem} the
ultrametric spaces
\(
(X_1,d|_{X_1\times X_1}),\)
\((X_2,d|_{X_2\times X_2}),\)
\((Y_1,\rho|_{Y_1\times Y_1}),\)
\((Y_2,\rho|_{Y_2\times Y_2})
\)
belong to ${\bf MCD}$ and the equalities
\begin{equation*}
    |X_i|=|Y_i|=2^{n_0-1},
\end{equation*}
\begin{equation}
    \label{ddaad4}
    |C(X_i)|=|C(Y_i)|=n_0
\end{equation}
hold for $i=1,2$ if \eqref{v1-s} holds with $n=n_0$.
The induction hypothesis and equalities \eqref{ddaad4} imply that all spaces
\(
(X_1,d|_{X_1\times X_1}),
\)
\(
(X_2,d|_{X_2\times X_2}),
\)
\(
(Y_1,\rho|_{Y_1\times Y_1})
\) and \(
(Y_2,\rho|_{Y_2\times Y_2})
\)
are isometric. Let
\(
\Phi_1:X_1\to Y_1
\)
and
\(
\Phi_2:X_2\to Y_2
\)
be the corresponding isometries and let a mapping 
\(
\Phi:X\to Y
\)
be defined as
\begin{equation}
    \label{ddaad5}
\Phi(p):=
\begin{cases}
\Phi_1(p), & \text{if }p\in X_1,\\
\Phi_2(p), & \text{if }p\in X_2.
\end{cases}
\end{equation}
The mapping $\Phi:X\to Y$ is an isometry of $(X,d)$ and $(Y,\rho)$ if and only if the equality
\begin{equation}
    \label{ddaad6}
d(p,q)=\rho(\Phi(p),\Phi(q))
\end{equation}
holds for all $p,q\in X$.
Let us prove equality \eqref{ddaad6}. If
\(
p,q\in X_1 \) or
\( p,q\in X_2,
\)
equality \eqref{ddaad6} follows from \eqref{ddaad5}, since
\(
\Phi_1:X_1\to Y_1
\) and \(
\Phi_2:X_2\to Y_2
\)
are isometries. Suppose that $p$, $q$ belong to distinct parts $X_1,$ $X_2$
of $G_X$. Then, using  equality \eqref{ddaad5}, we see that
\(
\Phi(p)\), \(\Phi(q)
\)
belong to distinct parts $Y_1,$ $Y_2$ of $G_Y$. Now,
Definitions \ref{edrtaa}, \ref{okiu} imply the equalities 
\begin{equation}
    \label{ttr1}
d(p,q)=\operatorname{diam} X,
\end{equation}
\begin{equation}
    \label{ttr2}
\rho(\Phi(p),\Phi(q))
=\operatorname{diam} Y. 
\end{equation}
By Lemma \ref{UT2Lemma3} the equalities
\begin{equation*}
C(X)=D(X),\quad C(Y)=D(Y) 
\end{equation*}
hold and, consequently, we have
\begin{equation*}
    D(X)=D(Y)
\end{equation*}
by equality \eqref{m2}.
The last equality and the definition of diameter of metric spaces imply
\[
\operatorname{diam} X
=
\max\{t:t\in D(X)\}
=
\max\{t:t\in D(Y)\}
=
\operatorname{diam} Y.
\]
Equality \eqref{ddaad6} follows from \eqref{ttr1}, \eqref{ttr2}.

Thus the implication
\(
(iv)\Rightarrow(i)
\)
is valid.
\end{proof}

 Theorem \ref{m0} implies the following corollary.

\begin{cor} \label{UT2Cor7}
Let $n\geq 1$ be an integer number and let $(X,d)$ and $(Y,\rho)$ be ultrametric spaces such that
\begin{equation}\label{dcfpq}
|X| = |Y| = 2^n \quad \text{\it and} \quad  |C(X)| = n+1.
\end{equation}
Then the equality
    \begin{equation}
        \label{hi1}
    C(Y) = C(X)
        \end{equation}
holds if and only if $(X,d)$ and $(Y,\rho)$ are isometric.
\end{cor}

\begin{proof}
Equality \eqref{hi1} is trivially valid if $(X,d)$ and $(Y,\rho)$ are isometric.

To prove that this equality implies the isometricity of $(X,d)$ and $(Y,\rho)$, we note that \eqref{dcfpq}  and \eqref{hi1} give us 
\begin{equation*}
    |X|=|Y|=2^n \quad \text{and} \quad |C(X)|=|C(Y)|=n+1.
\end{equation*}
Hence, the equalities
\[
|X|=2^{\,|C(X)|-1} \quad \text{and} \quad |Y|=2^{\,|C(Y)|-1}
\]
hold, and, consequently, $(X,d)$ and $(Y,\rho)$ are {\bf MCD}-spaces.
Thus $(X,d)$ and $(Y,\rho)$ are isometric by Theorem~\ref{m0}.
\end{proof}

\begin{rem}
Corollary \ref{UT2Cor7} was formulated as  Conjecture~4.3 in \cite{DovgosheyRovenska2026}. Thus this conjecture is valid.
\end{rem}

\begin{prop}
     \label{UT2Cor8}
Let $A$ be a finite subset of $\mathbb{R}^+$ and let $0 \in A$. Then, up to isometry, there exists a unique ultrametric space $(X,d)\in {\bf MCD}$ such that
\begin{equation}\label{dse}
    D(X) = C(X) = A.
\end{equation}
\end{prop}

\begin{proof}
    If $A = \{0\}$ holds, then any one-point ultrametric space $(X,d)$ serves the required purpose.

Let the inequality $|A| \ge 2$ hold.
In this case there is a numbering the points of $A$ such that
\begin{equation*}
A=\{a_0,\ldots,a_n\}
\end{equation*}
and
\begin{equation*}
a_0>a_1>\ldots>a_n=0.
\end{equation*}

    Let us consider a perfect binary tree  $T=T(r)$ such that
the level of every leaf of $T(r)$ is exactly $n$
    and define a labeling $l: V(T)\to \mathbb{R}^+$ such that 
\begin{equation}\label{fgqnhy}
\bigl(\operatorname{lev}_{T(r)}(u)=k\bigr)
\iff
\left(
l(u)=a_k
\right)
\end{equation}
for all $u\in V(T)$
and $a_k\in\{a_0,\ldots,a_n\}.$

Then $T=T(r,l)$ satisfies statement $(i)$ of Theorem~\ref{mmre} and, consequently, there exists 
\(
(X,d)\in {\bf MCD}
\)
such that
\(
T_X=T_X(r_X,l_X)
\)
and
\(
T=T(r,l)
\)
are isomorphic as labeled rooted trees. By Theorem~\ref{m0} the ultrametric space $(X,d)$ is unique up to isometry.
Thus, to complete the proof it suffices to show that equalities~\eqref{dse} hold.

Lemma~\ref{UT2Lemma2} and equivalence~\eqref{fgqnhy}  imply the equality
\begin{equation*}
D_p(X)=\{a_0,\ldots,a_n\},
\end{equation*}
where $D_p(X):=\{d(x,p):x\in X\}$
for each
$p\in X$.
Consequently, we have
\begin{equation}\label{eq:CX}
C(X)=\{a_0,\ldots,a_n\}
\end{equation}
by Lemma~\ref{xute}. Since $(X,d)$ is a ${\bf MCD}$-space, Lemma \ref{UT2Lemma3} implies the equality
\begin{equation}\label{eq:CD}
C(X)=D(X).
\end{equation}

Now \eqref{eq:CX} and \eqref{eq:CD} give us \eqref{dse}.
\end{proof}

Proposition \ref{UT2Cor8} implies the following corollary.

\begin{cor}\label{ee}
Let $A$ be
a finite subset of $\mathbb{R}^{+}$ such
that $0\in A$. Then there exists
a finite ultrametric space $(X,d)$
that satisfies the
equalities
\[
D(X)=C(X)=A.
\]
\end{cor}

\begin{rem} The above Corollary \ref{ee} was formulated in~\cite{DovgosheyRovenska2026} as Conjecture~4.4.
Thus this conjecture is valid.
\end{rem}

\section{Problems and expected results}\label{seq5}

Following \cite[p.~38]{Dez-Dez} we say that a metric space $(X,d)$ is {\it homogeneous} if, for any two points $x,y\in X$, there exists a self-isometry $F : X \to X$ satisfying the equality $F(x)=y$.

The authors believe the following conjecture is true.

\begin{con}\label{llp} Let $(X,d)$ be a ${\bf MCD}$-space and let $\operatorname{Is}(X)$ be the set of all self-isometries of $(X,d)$. Then the following statements hold.
\begin{enumerate}
    \item[(i)] $(X,d)$ is homogeneous.
    \item[(ii)] The equality
    \begin{equation}
        \label{lli}
      \log_2  |\operatorname{Is}(X)| = |X|-1
        \end{equation}
    holds.
\end{enumerate}
\end{con}

\begin{con}[Prove or disprove]
Let $(X,d)$ be a finite homogeneous ultrametric space. Then equality \eqref{lli} holds if and only if $(X,d)\in {\bf MCD}$.
\end{con}

For each metric space $(X,d)$, the set $\operatorname{Is}(X)$ of all self-isometries of $(X,d)$ is a group with respect to the operation of composition of self-isometries of $(X,d)$.

\begin{prob}
    \label{gapl}
Describe the structure of the groups $\operatorname{Is}(X)$ of ${\bf MCD}$-spaces $(X,d)$ up to isomorphism.
\end{prob}

The following problem is closely related with Corollary \ref{ee}.

\begin{prob}
Let $b_0,\ldots,b_n$ be a finite sequence of positive integer numbers. Find conditions under which there exist a finite ultrametric space $(X,d)$ having the center of distances
\[
C(X)=\{c_0,\ldots,c_n\}, \quad 0=c_0<\cdots<c_n,
\]
such that the equality
\begin{equation}
    \label{gaplll}
b_i=\left|\{x\in X:d(p,x)=c_i\}\right|
\end{equation}
holds for all $p\in X$ and $i\in\{0,\ldots,n\}.$
\end{prob}

\begin{con}
    \label{tcedg}
    Let $(X,d)$ be a finite ultrametric space
with the center of distances
\(
C(X)=\{c_0,\ldots,c_n\}, \) where \(0=c_0<\cdots<c_n.
\)
Then following statements are equivalent.
\begin{enumerate}
    \item  
    Let
\(
b_0=1,\) \(b_1= 2^{1-1},\) \(
b_2=2^{2-1},\ldots,\) \(b_n=2^{n-1}.
\)
Then equality \eqref{gaplll} holds 
for all $p \in X$ and $i \in \{0,\ldots,n\}$.

\item The space $(X,d)$ is a ${\bf MCD}$-space with $|X|=2^n$.
\end{enumerate}
\end{con}

The next conjecture describes the structure of representing trees of finite
homogeneous ultrametric spaces up to isomorphism of rooted trees.

\begin{con} Let $T=T(r)$ be a finite rooted tree. Then
the following statements are equivalent.
\begin{enumerate}
    \item[(i)] For every $u\in V(T)$ we have
    \(
        \delta^+_{T_{(r)}}(u)\neq 1
    \)
    and, in addition, the equality
    \[
    \operatorname{lev}_{T_{(r)}}(v_1)
        =
    \operatorname{lev}_{T_{(r)}}(v_2)
    \]
    implies
    \[
        \delta^+_{T_{(r)}}(v_1)
        =
        \delta^+_{T_{(r)}}(v_2)
    \]
    for all $v_1,v_2\in V(T)$.

    \item[(ii)] There exists a finite homogeneous ultrametric space
    $(X,d)$ for which the representing tree
    \(
        T_X=T_X(r_X,l_X)
    \)
    and $T=T(r)$ are isomorphic as rooted trees.
\end{enumerate}
\end{con}

\subsection*{Funding}
Oleksiy Dovgoshey was supported by grant 367319 from the Research Council of Finland.

\bibliographystyle{plain}
\bibliography{bib2020.07}

\end{document}